\documentclass[reqno, a4paper, draft]{amsart}
\usepackage{amssymb, amscd, amsfonts,  amsthm}

\usepackage[mathscr]{eucal}
\usepackage{graphicx}
\usepackage[all,poly]{xy} 
 \usepackage[all]{xy} 
\usepackage{a4wide}

\newtheorem{theorem}{Theorem}[section]
\newtheorem{lemma}[theorem]{Lemma}

\newtheorem{proposition}[theorem]{Proposition}

\theoremstyle{definition}
\newtheorem{definition}[theorem]{Definition}
\newtheorem{remark}[theorem]{Remark}

\newtheorem{example}[theorem]{Example}

\newcommand{\restrict}{\,{\mathbin{\vert\mkern-0.3mu\grave{}}}\,}

\newcommand{\remove}[1]{}

\DeclareMathOperator{\Rn}{{\mathbb R^{\it n}}}
\DeclareMathOperator{\R2}{{\mathbb R^{2}}}

\DeclareMathOperator{\riesz}{\mathcal R}
\DeclareMathOperator{\conv}{\rm conv}

\DeclareMathOperator{\I}{[0,1]}
\DeclareMathOperator{\cube}{[0,1]^{\it n}}
\DeclareMathOperator{\aff}{\rm aff}

\DeclareMathOperator{\interior}{\rm int}
\DeclareMathOperator{\relint}{\rm relint}

 \title[Severi-Bouligand tangents, Frenet frames, Riesz spaces]
{Severi-Bouligand tangents, 
Frenet frames  and Riesz spaces}

\author{\sc Leonardo Manuel Cabrer and Daniele Mundici}

 \thanks{This research was supported by a 
 Marie Curie Intra European Fellowship 
 within the 7th European Community Framework 
 Program (ref. 299401-FP7-PEOPLE-2011-IEF)} 
 
\address[L.M. Cabrer]{Department of Statistics,
Computer Science and Applications,  ``Giu\-sep\-pe Parenti''\\ 
University of Florence\\
Viale Morgagni 59 --
50134\\ Florence \\
Italy}
\email{ l.cabrer@disia.unifi.it }

\address[D. Mundici]{Department of
Mathematics and Computer Science  ``Ulisse Dini'' \\
University of Florence\\
Viale Morgagni 67/A \\
I-50134 Florence \\
Italy}
\email{ mundici@math.unifi.it }

\date{\today}

\begin{document}


   \begin{abstract}
A compact set $X\subseteq \R2$
   has an outgoing  Severi-Bouligand  tangent unit vector  $u$ 
   at  some point $x\in X$    iff  some principal
   quotient of  the Riesz space  $\riesz(X)$
   of  piecewise linear functions on $X$
   is not archimedean.
To generalize this preliminary result, we
extend the classical definition of Frenet $k$-frame to any
sequence  $\{x_i\}$  of points in $\Rn$ 
converging to a point $x$, in such a way that when
the $\{x_i\}$ arise as sample points of a smooth curve $\gamma,$
the Frenet $k$-frames of  $\{x_i\}$ and of  $\gamma$
 at $x$  coincide.
Our method of computation of Frenet frames via sample sequences 
of $\gamma$ does not require
 the knowledge of any higher-order derivative
of $\gamma$.  Given a compact set 
$X\subseteq \Rn$ and a point $x\in\Rn$, a 
 Frenet $k$-frame  $u$ is  said to be
a  {\it tangent} of $X$ at $x$   if
$X$ contains a sequence $\{x_i\}$ 
converging to $x$, 
 whose
  Frenet $k$-frame   is  $u$.
We   prove that   $X$ has an outgoing 
  $k$-dimensional  tangent of $X$ iff
some principal  quotient of  $\riesz(X)$ is not
archimedean. 
If, in addition,  $X$ is convex, then
$X$ has no outgoing tangents iff  it is a polyhedron. \end{abstract}

 \maketitle

\section{Introduction}
In \cite[\S 53, p.59 and p.392]{sev0} and
 \cite[\S1, p.99]{sev},
 Severi defined  (outgoing) 
tangents  
 of   arbitrary subsets of the euclidean space 
$\Rn$.
 Subsequently  and independently, 
  Bouligand  defined the same notion
  \cite[p.32]{bou}, which today is widely known
  as ``Bouligand tangent''.  Throughout we will adopt the
  following equivalent definition, 
 where  $||\cdot||$ denotes  euclidean norm
 and $\conv(Y)$ is the convex hull of  $Y\subseteq \Rn$:

 \begin{definition}\cite[pp.14 and 133]{mor}\,\,\,\,
 \label{definition:severi}
 Let $\emptyset\not=X\subseteq \Rn$ and
 $x\in \Rn$.
 A unit vector $u \in \Rn$
 is a  {\it Severi-Bouligand tangent  
of $X$ at $x$}  if $X$ contains a sequence  
$\{x_i\}$ such that
$x_i\not=x$ for all $i$,\,\,\,
 $\lim_{i\to \infty } x_i = x,\,\,$ 
\,\,and 
$\lim_{i\to \infty } {(x_i-x)}/{||x_i-x||} =u.$
 If  for some
$\mu> 0,\,\,\,$   $\conv(x,x+\mu u)\cap X=\{x\}$, 
we say that  
$u$  is  {\it outgoing}.  
   \end{definition}


For an equivalent algebraic handling of tangents,
in Section  \ref{Sec:Main} we introduce
the  Riesz space (=vector lattice)  
  $\riesz(X)$  of  piecewise linear   functions on 
 any nonempty compact set $X\subseteq \mathbb R^n$.
 When $n=2$, the geometric properties of $X$ are 
  immediately  linked
  to the  algebraic properties of  
  $\riesz(X)$ by the following elementary result
  (Lemma   \ref{lemma:busmun}):
  {\it  If  $\riesz(X)$
 has a  non-archimedean principal quotient
then 
$X$ has  an outgoing Severi-Bouligand  tangent. }

In Theorem \ref{theorem:main} we will extend
  this result, as well as its converse, to
  all $n$. To this purpose,
in Section \ref{section:frenet}    
we introduce the notion of
a  Frenet $k$-frame {\it of a sequence}
$\{x_i\}$
of points in $\Rn$, as the
natural generalization of the classical Frenet (Jordan)
$k$-frame \cite{kue, jor} of a curve $\gamma$. 
Specifically, 
if  the $x_i$ arise as sample points
of  a smooth curve $\gamma$ 
accumulating at some point $x$ of $\gamma$, then the 
Frenet $k$-frame
of $\{x_i\}$ coincides with the Frenet 
$k$-frame of $\gamma$
at $x$.  This is   Theorem \ref{theorem:eguali}. 
The  proof  {\it     
yields a method to calculate
the  Frenet $k$-frame of a 
$C^{k+1}$ curve  $\gamma$ at a point $x$  
without knowing the  derivatives
of any
  parametrization of  $\gamma$}:  
  one  just takes
a  sampling sequence  
$\{x_i\}$ of points of $\gamma$ converging to $x$,
 and then makes the 
 linear algebra calculations as in the 
 proof of the theorem.
To show the wide applicability of our method,
Example \ref{example:befana} provides a  
 curve $\gamma$ having no Frenet  $k$-frame at a point
$x$, but such that the Frenet $k$-frame of 
each sequence of points of $\gamma$
converging to $x$  exists and is
independent of the parametrization of $\gamma.$

In Section \ref{section:pwl}  we deal with
the relationship between the Frenet $k$-frame
$u=(u_1, \dots , u_k )$
of  a sequence
   $\{x_i\}$  in $\Rn$ 
  converging to $x$,
  and any simplex
   $\,\,T \subseteq \Rn$ 
containing  $\{x_i\}$.
 Theorem \ref{theorem:contains-u-simplex}
shows that
$\,T$ automatically
contains  the simplex
$\,\,\conv(x,x+\lambda_1u_1,\ldots, x+\lambda_1u_1
+\cdots+\lambda_ku_k),$
 for some 
$\lambda_1, \dots , \lambda_k>0.$
This elementary result will find repeated use in the rest of the paper.

As 
a $k$-dimensional generalization of 
the classical Severi-Bou\-li\-gand  tangents, 
we  then say that a
 Frenet $k$-frame  $u$ is  {\it tangent} at $x$ to a compact
set $X\subseteq \Rn$  if
$X$ contains a 
sequence
  $\{x_i\}$ 
converging to $x$, 
 whose
  Frenet $k$-frame   is  $u$.
Then Theorem \ref{theorem:main} 
provides  the desired
strengthening of
Lemma \ref{lemma:busmun}, 
showing  that   $X$ has no outgoing tangent
iff   
 {\sl  every principal ideal of $\mathcal R(X)$ 
  is an intersection of maximal ideals.}
  This latter property 
  is  considered  in the literature for various classes
  of structures:   
For  commutative noetherian rings  it
  is known as  ``von Neumann regularity'';
frames  having this property
are known as ``Yosida frames'',  \cite[2.1]{marzen}; 
 Chang MV-algebras  with this property
  are said to be ``strongly semisimple'',  \cite{busmun}.
  As a corollary of Stone
  representation (\cite[4.4]{luxzaa}),
   every boolean algebra 
   is strongly semisimple.

    Since   $\{+,-,\wedge,\vee\}$-reducts 
of  Riesz spaces with
    strong unit are lattice-ordered abelian groups with strong unit,
    and the latter are categorically equivalent to     MV-algebras,
    \cite[3.9]{mun-jfa}, following 
    \cite{busmun} 
    we  say that a Riesz space $R$ is {\it strongly semisimple}
if  every principal ideal of $R$ 
  is an intersection of maximal ideals of $R$. Equivalently,
  every principal quotient of $R$ is archimedean.
  A large class of examples of strongly semisimple
  Riesz spaces with totally disconnected maximal
  spectrum is immediately provided by hyperarchimedean
  Riesz spaces, \cite{balmar}. 
  At the other extreme, when $X$ is a polyhedron,
   $\mathcal R (X)$ is strongly semisimple,
 (see Proposition  \ref{proposition:Poly}). 

Using Theorem \ref{theorem:main}, in
 Theorem \ref{theorem:convex}
 we prove that a  nonempty compact {\it convex\/}
subset  $X\subseteq\Rn$
has no outgoing tangent
iff  $X$  has only finitely many extreme points iff   $X$ is a polyhedron.
This  shows the naturalness
 of Definition  \ref{definition:tangent} 
  of ``outgoing tangent''  as a
 $k$-dimensional  extension of
 the classical  Severi-Bouligand tangent.
Counterexamples   of
   Theorem \ref{theorem:convex} are  easily found
in case $X$ is not convex (see Example
\ref{counterexample}).

 The  only prerequisite  for this paper is a working knowledge of
 elementary polyhedral topology  (as given, e.g.,  by
 the  first chapters of \cite{sta}), 
 and of
 the classical  Yosida (Kakutani-Gelfand-Stone)
 correspondence between points of $X$ and maximal ideals
 of the  Riesz space $\mathcal R(X)$. See  \cite{luxzaa} for a
 comprehensive account.

 \section{The Frenet frame of a sequence  $\{x_i\}\subseteq \Rn$} 
 \label{section:frenet}

Given two sequences  $\{p_i\},\{q_i\}\subseteq \mathbb R$,
by writing
$
\lim_{i\to \infty}{p_i}/{q_i} =r
$
 we understand that  $q_i\not=0$ for each $i$, and 
  $\,\,\lim_{i\to \infty}p_i/q_i\,\,$  exists and equals $r$.

For any vector $y\in\Rn$ and linear subspace
$L$ of $\Rn$,
the orthogonal projection of $y$ onto  $L$
is denoted  $$\mathsf{proj}_L(y).$$ 
 
 For our generalization of Severi-Bouligand tangents
we first extend   Definition
 \ref{definition:severi},  replacing the unit vector $u\in \Rn$ therein
 by a  $k$-tuple  
 $\{u_1,\ldots,u_k\}$ of pairwise orthogonal
 unit vectors in  $\Rn$.

  \begin{definition}
\label{definition:eigen-frame}
Given a sequence  $\sigma=\{x_i\}$ 
of points in $\Rn$ converging to $x,$  
and a $k$-tuple $(u_1,\ldots,u_{k})$
of pairwise orthogonal unit vectors
 in $\Rn$, we say:
 
\begin{itemize}
\item {\it $u_1$ is the  Frenet $1$-frame of $\sigma$}
 if $u_1=\lim_{i\to\infty} (x_i-x)/||x_i-x||$; 
 
%

\smallskip 
\item {\it $(u_1,\ldots,u_k)$ is the 
 Frenet $k$-frame of $\sigma$} if
 $(u_1,\ldots,u_{k-1})$ is the  Frenet $(k-1)$-frame of $\sigma$, and
$$
u_{k}=\lim _{i\to \infty}\frac{
x_i-x-
\mathsf{proj}_{\mathbb Ru_1+\cdots+\mathbb Ru_{k-1}}(x_i-x)
}
{||x_i-x-
\mathsf{proj}_{\mathbb Ru_1+\cdots+\mathbb Ru_{k-1}}(x_i-x)||}.
$$ 
\end{itemize}
 \end{definition}

 Following \cite{kue}, for  $[a,b]\subseteq \mathbb R$  an interval, 
suppose
$\phi\colon [a,b]\to \Rn$ is a $C^k$ function such that
for all $a\leq t < b$, the $k$-tuple of vectors 
$\,\,
(\phi'(t),\phi''(t),\ldots.,\phi^{(k)}(t))
\,\,$
forms a  linearly independent set
 in $\Rn$. 
Then the  Gram-Schmidt 
process
 yields an orthonormal 
$k$-tuple  $(v_1(t),\ldots,v_k(t))$,  called the {\it Frenet $k$-frame}
 of $\phi$ at $\phi(t)$.
 
 \smallskip 
 
The terminology of Definition~\ref{definition:eigen-frame} is  justified by the following result:

 \begin{theorem}
 \label{theorem:eguali}
Suppose
$\phi\colon [a,b]\to \Rn$ is a $C^{k+1}$ function.
Let  $a\leq t_0<b$ be
 such that the vectors
 $\phi'(t_0),\phi''(t_0),\ldots.,\phi^{(k)}(t_0)$ are linearly independent.
Then for every  sequence 
$\,t_1,t_2,\ldots$   in $[t_0,b]\setminus\{t_0\}$
 converging to $t_0$, the Frenet $k$-frame of $\{\phi(t_i)\}$ exists and is equal to the
Frenet  $k$-frame of $\phi$  at  $\phi(t_0).$
 \end{theorem}
 
 \begin{proof}
 We can write
\begin{equation}\label{Eq:Taylor}
\phi(t)=\phi(t_0)+\phi'(t_0)(t-t_0)+\frac{\phi^{''}(t_0)}{2}(t-t_0)^2
+\cdots+\frac{\phi^{(k)}(t_0)}{k!}(t-t_0)^k+R(t)\,,
\end{equation}
where the remainder $R\colon[a,b]\to \Rn$
satisfies 
\begin{equation}
\label{equation:rest}
 ||R(t)||\leq M(t-t_0)^{k+1}\,\,\,\,\mbox{for some } 0\leq  M\in \mathbb R.
 \end{equation}
Let $(v_1,\ldots,v_k)$ be the Frenet 
$k$-frame of $\phi$ at $\phi(t_0)$. 
Then  $v_1=\phi'(t_0)/||\phi'(t_0)||,$ and for each $1<j\leq k,$ 
$$v_j=\frac{\phi^{(j)}(t_0)-\mathsf{proj}_{\mathbb Rv_1+\cdots
+\mathbb Rv_{j-1}}(\phi^{(j)}(t_0))}{||\phi^{(j)}(t_0)-\mathsf{proj}_{\mathbb Rv_1
+\cdots+\mathbb Rv_{j-1}}(\phi^{(j)}(t_0))||}.
$$

By induction on $1\leq j\leq k$ we will prove
that the Frenet $j$-frame $(u_1,\ldots,u_j)$ of
the sequence
$\{\phi(t_i)\}$ (exists and)  coincides with the
Frenet $j$-frame  $(v_1,\ldots, v_j)$
of $\phi$ at $\phi(t_0)$.

\medskip
 \noindent{\it Basis:} Since $||\phi'(t_0)||\neq0$, 
for  all suitably large  $i$  we have $\phi(t_i)\neq\phi(t_0)$ 
and
\begin{eqnarray*}
u_1&=&\lim_{i\to\infty}\frac{\phi(t_i)-\phi(t_0)}
{||\phi(t_i)-\phi(t_0)||}\\
{}&=&\lim_{i\to\infty}\frac{(\phi(t_i)-\phi(t_0))/(t_i-t_0)}
{||(\phi(t_i)-\phi(t_0))/(t_i-t_0)||}\\
{}&=&\frac{\lim_{i\to\infty}(\phi(t_i)-\phi(t_0))/(t_i-t_0)}
{||\lim_{i\to\infty}(\phi(t_i)-\phi(t_0))/(t_i-t_0)||}\\
{}&=&\frac{\phi'(t_0)}{||\phi'(t_0)||}\\[0.2cm]
&=&v_1.
\end{eqnarray*}

\noindent{\it Induction Step:} 
By induction hypothesis, for  each  $1\leq j<k$ 
the $j$-tuple  $(v_1,\ldots, v_j)$ coincides with
the   Frenet 
 $j$-frame 
 $(u_1,\ldots,u_j)$
 of the sequence $\{\phi(t_i)\}$.
 Let  the  linear subspace
 $S_j$  of $\Rn$ be defined by
 $$
 S_j=
  \mathbb Ru_1+\cdots+\mathbb Ru_j=
 \mathbb Rv_1+\cdots+\mathbb Rv_j=\mathbb R\phi'(t_0)+\cdots+\mathbb R\phi^{(j)}(t_0).
 $$
From \eqref{equation:rest} we have
\begin{equation}
\label{equation:remainder}
\frac{||R(t)-\mathsf{proj}_{S_j}(R(t))||}{(t-t_0)^{j+1}}\leq  M(t-t_0)^{k-j}.
\end{equation}
For each $l=j+1,\ldots,k$ let us define the vector
  $\alpha_l\in \Rn$  by
\begin{equation}
\label{equation:alpha}
\alpha_l=\frac{\phi^{(l)}(t_0)-\mathsf{proj}_{S_j}(\phi^{(l)}(t_0))}{l!},
\end{equation}
whence in particular, 
 $$
 ||\alpha_{j+1}||=\frac{||\phi^{(j+1)}(t_0)-\mathsf{proj}_{S_j}(\phi^{(j+1)}(t_0))||}{(j+1)!}\neq 0.
 $$

\noindent   
By   (\ref{Eq:Taylor}), 
\begin{multline}\label{Eq:middle}
\phi(t_i)-\phi(t_0)-\mathsf{proj}_{S_j}(\phi(t_i)-\phi(t_0))=\\
\alpha_{j+1}(t_i-t_0)^{j+1}+\cdots+\alpha_{k}(t_i-t_0)^k
+R(t_i)-\mathsf{proj}_{S_j}(R(t_i)). 
\end{multline}
From \eqref{equation:remainder}-\eqref{Eq:middle} we get
\begin{align*}
& \,\,\,\,\,\,\,\,\,u_{j+1}=\lim_{i\to\infty}\frac{\phi(t_i)-\phi(t_0)
-\mathsf{proj}_{S_j}(\phi(t_i)-\phi(t_0))}{||\phi(t_i)
-\phi(t_0)-\mathsf{proj}_{S_j}(\phi(t_i)-\phi(t_0))||}\\[0.15cm]
&=\lim_{i\to\infty}\frac{\alpha_{j+1}(t_i-t_0)^{j+1}
+\cdots+\alpha_{k}(t_i-t_0)^k+R(t_i)-\mathsf{proj}_{S_j}(R(t_i))}
{||\alpha_{j+1}(t_i-t_0)^{j+1}+\cdots
+\alpha_{k}(t_i-t_0)^k+R(t_i)-\mathsf{proj}_{S_j}(R(t_i))||}\\[0.15cm]
&=\lim_{i\to\infty}\frac{{\sum_{l=j+1}^{k}\alpha_{l}(t_i-t_0)^{l-(j+1)}
+(R(t_i)-\mathsf{proj}_{S_j}(R(t_i)))}\cdot{(t_i-t_0)^{-(j+1)}}}
{||{\sum_{l=j+1}^{k}\alpha_{l}(t_i-t_0)^{l-(j+1)}
+(R(t_i)-\mathsf{proj}_{S_j}(R(t_i)))}
\cdot{(t_i-t_0)^{-(j+1)}||}}\\[0.15cm]
&=\frac{\alpha_{j+1}}
{||\alpha_{j+1}||}=\frac{\phi^{(j+1)}(t_0)
-\mathsf{proj}_{S_j}(\phi^{(j+1)}(t_0))}{||\phi^{(j+1)}(t_0)
-\mathsf{proj}_{S_j}(\phi^{(j+1)}(t_0)) ||}=v_{j+1}.
\end{align*}

\medskip
\noindent
This concludes the proof.   
\end{proof}

\begin{remark}
The assumption  $\phi\in C^{k+1}$ can be relaxed
to  $\phi\in C^{k}$, so long as  the $k$th
Taylor remainder  $R(t)$ satisfies \eqref{equation:rest}.
\end{remark}


\begin{remark}
Theorem  \ref{theorem:eguali}
yields a method to calculate
the  Frenet $k$-frame of a 
$C^{k+1}$ curve,
not involving higher-order derivatives, 
but taking instead
a  sampling  sequence  
$\{x_i\}$ of points on the curve,
 and then making the elementary  
 linear algebra calculations in the proof above.  
 \end{remark}

 The wide applicability of this method
is shown by the  following example: 

\begin{example}
\label{example:befana}
Let $\phi\colon[0,1]\to \mathbb{R}^2$ be defined by $\phi(x)=(x,x^3)$. 
Then  $\phi'(0)=(1,0)$ and $\phi^{''}(0)=(0,0)$. 
The Frenet $1$-frame of $\phi$ at $(0,0)$ is the vector $(1,0)$, but
$\phi$ has no Frenet $2$-frame at $(0,0)$. 
And yet,  letting $\mathbb R(1,0)$ denote the linear subspace of
$\mathbb R^2$ given by the 
$x$-axis,
every sequence $\{t_i\}\in [0,1]\setminus \{0\}$ converging to $0$ satisfies 
\[
\lim_{i\to\infty}\frac{\phi(t_i)-\phi(0)-\mathsf{proj}_{\mathbb R(1,0)}(\phi(t_i)-\phi(0))}
{||\phi(t_i)-\phi(0)-\mathsf{proj}_{\mathbb R(1,0)}(\phi(t_i)-\phi(0))||}
=\lim_{i\to\infty}\frac{(0,t_i^3)}{||(0,t_i^3)||}=(0,1).
\]
We have shown:
{\it There exist a curve $\gamma$ having no Frenet  $k$-frame at a point
$x$, but the Frenet $k$-frame of 
every sequence of points of $\gamma$
converging to $x$  exists and is
independent of the parametrization of $\gamma.$}  
\end{example}


\begin{example}
\label{example:genoveffa}
While under the   hypotheses of Theorem
 \ref{theorem:eguali} the Frenet $k$-frames
 of any two sampling sequences of a curve $\gamma$
 at a point $x\in \gamma$ are equal, the map
 $\psi(x)=(x,x^2\sin(1/x))\colon [0,1]\to \mathbb R^2$  (with the proviso that
 $\psi(0)=(0,0)$), yields
 an   example
 of a curve   $\gamma$ that is not
  $C^{2}$ and has two sequences
$\{x_i\}$ and $\{y_i\}$ of points of $\gamma$ both converging to the
same point $(0,0)$ of $\gamma$,  but having different Frenet $2$-frames.
\end{example}


 \section{Simplexes and Frenet frames} 
 \label{section:pwl}

Fix $n=1,2,\ldots $.
For any  subset $E$ of the euclidean space 
$\mathbb R^n,$ the {\it convex hull}
$\,\,\,{\rm conv}(E)$  is the set of all  {\it convex
combinations} of elements of $E$.  
We say that $E$ is  {\it convex} if $E=\conv(E).$
For any  subset $Y$ of $\mathbb R^n$,  
  the   {\it affine hull}  $\aff(Y)$ 
of  $Y$ is the set of all {\it affine combinations} in  $\mathbb R^n$
of elements of $Y$. 
A set  $\{y_1,\ldots,y_m\}$ of  points in
$\mathbb R^n$ is said to be
{\it affinely independent}  if  none of its  elements 
 is  an affine combination of the remaining elements.
The {\it relative interior}\,\, $\relint(C)$
of a convex set $C\subseteq \mathbb R^n$
is the interior of $C$ in the affine hull  
 of $C$.
 For $\,\,0\leq d \leq n$, a {\it d-simplex}  $T$
in $\mathbb R^{n}$ is the
convex hull $\, \conv(v_{0},\ldots,v_{d})\,$ of $d+1$ affinely
independent points in  
$\mathbb R^n$.  
The {\it  vertices}
 $v_{0},\ldots,v_{d}$ are uniquely determined by $T$.
 A {\it face} of $T$ is the convex hull of a subset $V$ of vertices
 of $T$.  If the cardinality of $V$ is  $d$, then $V$   is
said to be a {\it facet} of $T$.
 
 The {\it positive cone of  $Y\subseteq \Rn$ at a point $x\in Y$} is the set
\begin{equation}
\label{equation:cone}
{\rm Cone}(Y,x)=\{y\in \Rn\mid x+\rho (y-x)\in Y
\mbox{ for some } \rho>0 \}.
\end{equation}
When  $T$ is a simplex,  ${\rm Cone}(T,x)$ is closed. 
If  $F$ is a face of $T$ and
 $x\in\relint(F)$ then for each $y\in F$ we have 
 \begin{equation}
 \label{Eq:Cone}
{\rm Cone}(T,x)=\aff(F)+{\rm Cone}(T,y).
\end{equation}
In particular, if $x\in\relint(T)$  then ${\rm Cone}(T,x)=\aff(T)$.

\medskip

 \begin{lemma}
 \label{Lem:DesInc}
Suppose $T\subseteq \Rn$ is a simplex
 and $F$ is a face of $T$.
 \begin{itemize}
 \item[(a)] 
If  $S$ is an arbitrary simplex contained in $T,$
and  $F\cap \relint(S)\not=\emptyset,$
then  $S$ is contained in $F$.

 \item[(b)]
A point $z$ lies in
$\relint(F)$ iff
$F$ is the smallest
 face of $T$ containing $z$.

\end{itemize}
\end{lemma} 
 
\begin{proof} 
(a) Let  $F_1,\ldots,F_u$ be the  facets 
 of $T$, with their respective affine hulls
$H_1,\ldots,H_u.$  Each  $H_j$  is the boundary of the closed
half-space  $H_j^+\subseteq T$ and of the other
closed half-space $H_j^-.$  Without loss of generality,
$F_1,\ldots,F_t$ are the facets of $T$ containing $F$. Then
$
\aff(F) =  H_1\cap\cdots \cap H_t$ and
$F= (H_{t+1}^+\cap\cdots\cap H_u^+)\cap \aff(F).
$
By way of contradiction, suppose  $x\in F\cap \relint(S)$
and  $y\in S\setminus F.$  For some $\epsilon>0$ the segment
$\conv(x+\epsilon(y-x), x-\epsilon(y-x))$ is contained in $S$.
For some hyperplane $H\in \{H_1,\ldots,H_t\}$ 
the point $y$ lies  in the open half-space
$\,\,\,\interior(H^+)= \Rn\setminus H^-,$  where
``int'' denotes  topological interior.
Now $x+\epsilon(y-x)\in \interior(H^+)$ and
$x-\epsilon(y-x)\in \interior(H^-),$ whence
$x-\epsilon(y-x)\notin T,$ which contradicts 
$S\subseteq T$.

(b)  This  easily follows from (a).
\end{proof}

 \begin{proposition}
\label{proposition:filterbase}
Let  $\,x\in \Rn$ and $\,u_1,\ldots,u_m$ be linearly independent vectors
in $\Rn$.  Let $\lambda_1,\,\mu_1,\ldots,\lambda_m,\,\mu_m>0.$ 
Then the intersection of the two
$m$-simplexes
$\conv(x,x+\lambda_{1}u_{1},
           \,\ldots, \,\, x+\lambda_{1}u_{1}+
           \cdots +\lambda_{m}u_{m})\,\,$ and  $\,\,\conv(x, x+ \mu_{1}u_{1},
           \,\ldots, \,\,x+ \mu_{1}u_{1}+
           \cdots +\mu_{m}u_{m})$ is 
an $m$-simplex 
of the form $\conv(x,
          x+ \nu_{1}u_{1},
           \,\ldots, \,\, 
           x+\nu_{1}u_{1}+
           \cdots +\nu_{m}u_{m})$
for uniquely determined  real numbers $\nu_1,\ldots,\nu_m>0$.
\end{proposition}

\begin{proof}
We argue by induction on $t=1,\ldots,m$. 
The cases $t=1,2$ are trivial.
Proceeding inductively,
for any simplex
 $W =\conv(x,x+\theta_{1}u_{1},
           \,\ldots, \,\, x+\theta_{1}u_{1}+
           \cdots +\theta_{t}u_{t})$, 
           let  $W' =\conv(x, x+ \theta_{1}u_{1},
           \,\ldots, \,\,x+ \theta_{1}u_{1}+
           \cdots +\theta_{t-1}u_{t-1})$
and
 $W'' =\conv(x, x+\theta_{1}u_{1},
           \,\ldots, \,\,x+ \theta_{1}u_{1}+
           \cdots +\theta_{t-2}u_{t-2})$.
By \eqref{Eq:Cone}, for each $y\in W'\setminus W''$ the half-line
 from $y$ in direction $u_t$ intersects  $W$
 in a segment  $\conv(y,y+\gamma u_t)$ for some
 $\gamma >0$ depending on $y$.
Now let          
                   $$
U_t= {\rm conv}(x,
          x+ \lambda_{1}u_{1},
           \,\ldots, \,\, x+\lambda_{1}u_{1}+
           \cdots +\lambda_{t}u_{t}),
           $$
           $$
           V_t = {\rm conv}(x,
          x+ \mu_{1}u_{1},
           \,\ldots, \,\,x+ \mu_{1}u_{1}+
           \cdots +\mu_{t}u_{t}).
        $$
We then have 
          $$
U_{t-1}= U_t'= {\rm conv}(x,
          x+ \lambda_{1}u_{1},
           \,\ldots, \,\,
           x+ \lambda_{1}u_{1}+
           \cdots +\lambda_{t-1}u_{t-1}),
           $$
                  $$
V_{t-1}=V_t'=  {\rm conv}(x,
          x+ \mu_{1}u_{1},
           \,\ldots, \,\, 
           x+\mu_{1}u_{1}+
           \cdots +\mu_{t-1}u_{t-1}) ,
        $$
        and
           $$
U_{t-2}= U_t''= {\rm conv}(x,
          x+ \lambda_{1}u_{1},
           \,\ldots, \,\,
           x+ \lambda_{1}u_{1}+
           \cdots +\lambda_{t-2}u_{t-2}),
           $$
                  $$
V_{t-2}=V_t''=  {\rm conv}(x,
          x+ \mu_{1}u_{1},
           \,\ldots, \,\,
           x+ \mu_{1}u_{1}+
           \cdots +\mu_{t-2}u_{t-2}) .
        $$       
        
 \smallskip
 \noindent       
             By induction hypothesis,  
for uniquely determined $\nu_1,\ldots,\nu_{t-1}>0$
we can write
        $$
U_{t}'\cap V_{t}' = {\rm conv}(x,
          x+ \nu_{1}u_{1},
           \,\ldots, \,\, 
           x+\nu_{1}u_{1}+
           \cdots +\nu_{t-1}u_{t-1}).
        $$
        The point $z=x+\nu_{1}u_{1}+
           \cdots +\nu_{t-1}u_{t-1}$
           lies in $U_{t}'\setminus U_{t}''.$
 Let  $\eta_1$ be the largest $\eta$  such that
$z+\eta u_t$  lies in  $U_t$. 
Since $z\in V_{t}'\setminus V_{t}'',\,\,$ 
let similarly
$\eta_2$ be the largest $\eta$  such that
$z+\eta u_t$  lies in  $V_t$. 
As already noted at the beginning of this proof, 
the real number
$\nu_t=\min(\eta_1,\eta_2)$ is $>0$. 
Evidently,  
$\nu_t$ is the largest $\eta$
such that
$z+\eta u_t$ lies in $U_t\cap V_t.$
We conclude that 
$U_t\cap V_t = \conv(x,
           x+\nu_{1}u_{1},
           \,\ldots, \,\, 
           x+\nu_{1}u_{1}+
           \cdots + \nu_{t}u_{t})$.
\end{proof}

The following key result will find repeated use in the rest of this paper:

  \begin{theorem}
  \label{theorem:contains-u-simplex}
   Let
  $(u_1, \dots , u_k )$ be the  
  Frenet $k$-frame of  a sequence
   $\{x_i\}$  in $\Rn$ 
  converging to $x$. 
  Suppose 
 a simplex  $\,\,T \subseteq \Rn$ 
contains  $\{x_i\}$. Then $\,T$
contains  the simplex
$\,\,\conv(x,x+\lambda_1u_1,\ldots, x+\lambda_1u_1
+\cdots+\lambda_ku_k),$
 for some 
$\lambda_1,\dots, \lambda_k>0.$ 
\end{theorem}

\begin{proof}
We will  prove 
the following stronger statement:
\medskip

\noindent{\it Claim.} 
For each $l\in \{1,\ldots,k\}$ 
there exist    $\lambda_1,\ldots,\lambda_l>0$ such that: 
\begin{itemize}
\item[(i)]$\conv(x,x+\lambda_1u_1,\ldots, x+\lambda_1u_1+\cdots+\lambda_lu_l)\subseteq T$ and 

\smallskip
\item[(ii)]  letting  $F_l$ be the smallest  face of $T$ 
containing  the point
 $z_l=x+\lambda_1u_1+\cdots+\lambda_lu_l$\,\,
  (which  by {Lemma \ref{Lem:DesInc}(b)} is equivalent to  
$z_l\in  \relint(F_l)$),
we have the inclusion
 $\conv(x,x+\lambda_1u_1,\ldots, x+\lambda_1u_1+\cdots+\lambda_lu_l)\subseteq F_l$.
\end{itemize}

The proof is  by induction on $l=1,\ldots,k$.

\medskip

\noindent
{\it Basis Step $(l=1)$:}   
Since each $x_i$ is in $T$
then  $x+(x_i-x)/||x_i-x||\in{\rm Cone}(T,x)$. 
Since ${\rm Cone}(T,x)$ is closed, then   $x+u_1\in{\rm Cone}(T,x)$. From 
 \eqref{equation:cone} we obtain an $\epsilon>0$ 
 such that $x+\epsilon u_1\in T$.   
Let $\lambda_1=\epsilon/2$. Then
$\conv(x,x+\lambda u_1)\subseteq \conv(x,x+\epsilon u_1)\subseteq T$, and (i) follows. 
Let $F_1$ be the smallest  face of $T$ containing the point
$z_1=x+\lambda_1u_1$.  Evidently,  $z_1\in 
\,\relint(\conv(x,x+\epsilon u_1)).$
{ By  Lemma~\ref{Lem:DesInc}(b),}   $z_i\in \relint(F_1).$  
{ By Lemma~\ref{Lem:DesInc}(a),}
$F_1\supseteq \conv(x,x+\epsilon u_1)\supseteq
\conv(x,x+\lambda u_1).$ 
This proves (ii) and concludes the proof of the basis step.

\medskip

\noindent
{\it Induction Step:}  For  $1\leq l<k$,  induction  
yields  $\lambda_1,\ldots,\lambda_l>0$ such that, letting 
$C_l=\conv(x,x+\lambda_1u_1,\ldots, x+\lambda_1u_1+\cdots+\lambda_lu_l) \,\,\mbox{ and }z_l=x+\lambda_1u_1+\cdots+\lambda_lu_l,$
we have  $C_l\subseteq T$.   Further,
letting  $F_l$ be the smallest face of $T$ containing $z_l$,
we have   $C_l\subseteq F_l,$  whence 
 $\aff(C_l)
 =x+\mathbb{R} u_1+\cdots+\mathbb{R} u_{l}\subseteq\aff(F_l)$.  Since $z_l\in\relint(F_l)$ and $x_i-x\in
  {\rm Cone} (T,x)$,
 { from (\ref{Eq:Cone})} we obtain
 $$z_l+\frac{x_i-x-\mathsf{proj}_{\mathbb{R} u_1+\cdots+\mathbb{R} u_{l}}(x_i-x)}
{||x_i-x-\mathsf{proj}_{\mathbb{R} u_1+\cdots+\mathbb{R} u_{l}}(x_i-x)||}\in {\rm Cone}(T,z_l).$$ 
 ${\rm Cone}(T,z_l)$ is closed, because    
$z_l+u_{l+1}\in{\rm Cone}(T,z_l)$. 
By  \eqref{equation:cone}, 
 there exists $\epsilon>0$ such that 
 $z_l+\epsilon u_{l+1}\in T$, whence  
 $\conv(z_l, z_l+\epsilon u_{l+1})\subseteq T$.
Setting now
  $\lambda_{l+1}=\epsilon/2$ and $z_{l+1}=z_l+\lambda_{l+1}u_{l+1}$, 
  condition  (i) in the claim above  follows from the identity 
$$
\conv(x,x+\lambda_1u_1,\ldots, x+\lambda_1u_1+\cdots+\lambda_{l+1}u_{l+1})=\conv(C_{l}\cup \{z_{l+1}\})\subseteq T.$$

 \noindent
Let  $F_{l+1}$ be the smallest  face of $T$ 
containing the point  $z_{l+1}\in \relint(\conv(z_l,z_l+\epsilon u_{l+1})).\,\,\,$
{By  Lemma~\ref{Lem:DesInc}(b),} 
$z_{l+1}\in \relint(F_{l+1}).\,\,\,$
By {Lemma~\ref{Lem:DesInc}(a)},\,\,\, 
$$ F_{l+1}\supseteq \conv(z_l, z_l+\epsilon u_{l+1})
\supseteq \conv(z_l, z_l+\lambda_{l+1} u_{l+1}).$$
The minimality property of $F_l$  yields
 $F_l\subseteq  F_{l+1}$.  By induction hypothesis, 
$C_{l}\subseteq F_{l+1}$. 
In conclusion,
 $
 \conv(x,x+\lambda_1u_1,\ldots, x+\lambda_1u_1+\cdots+\lambda_{l+1}u_{l+1})=\conv(C_{l}\cup \{z_{l+1}\})\subseteq F_{l+1},
 $ 
 as required to prove (ii) and
to complete the proof.
\end{proof}

\section{Tangents of $X$,  principal ideals of 
$\mathcal R(X)$: the case  $X\subseteq \mathbb R^2$}
\label{Sec:Main}

For $k=1$ the following definition  
boils down to Definition
 \ref{definition:severi} of Severi-Bouligand tangent
 vector. 
 As in Definition
 \ref{definition:severi},
$X$ is {\it an arbitrary}  nonempty subset of $\Rn$.

\begin{definition}
\label{definition:tangent}
Let $X\subseteq \Rn$, 
 $x\in \Rn$ and $u=(u_1,\ldots,u_k)$ be 
a $k$-tuple
of pairwise orthogonal unit vectors 
in $\Rn$.
Then  $u$ is said to be a
  {\it   tangent  of $X$ at $x$}
if $X$ contains a 
sequence
  $\{x_i\}$ 
converging to $x$, 
 whose
  Frenet $k$-frame   is  $u$.  We say that $\{x_i\}$ 
  {\it determines} $u$.
We say that 
$u$ is {\it outgoing} if,
in addition,  there are  $\lambda_1,\ldots,\lambda_k>0$ such that 
the simplex
$C=\conv(x,x+\lambda_1u_1,\ldots, x+\lambda_1u_1+\cdots+\lambda_{k}u_{k})$ and its facet 
 $C'=\conv(x,x+\lambda_1u_1,\ldots, x+\lambda_1u_1+\cdots+\lambda_{k-1}u_{k-1})$ 
have the same intersection with $X$.
\end{definition}

The following elementary  material
on piecewise linear topology \cite{sta} is necessary to introduce
the Riesz space  $\mathcal R(X)$ of piecewise linear
functions on $X$.  In Theorem 	\ref{theorem:main} below,
the Frenet tangent frames of $X$ will be related
to the maximal and principal ideals of $\mathcal R(X)$.

A {\it polyhedron} $P$ in $\Rn$  is a finite union of simplexes
in $\Rn.$  $P$ need not be convex or connected.
Given a  polyhedron $P$,   a {\it triangulation}  of
$P$ is an (always finite)
 simplicial complex $\Delta$ such that $P=\bigcup\Delta$.
Every polyhedron has a triangulation, \cite[2.1.5]{sta}.
%
%
Given a rational polyhedron $P$ and triangulations $\Delta$ and $\Sigma$ of $P$, 
we say that $\Delta$ is a {\it subdivision} of $\Sigma$ if every
simplex of $\Delta$ is contained in a simplex of $\Sigma$.
Suppose an $n$-cube  $K\subseteq \Rn$ is contained in another
$n$-cube  $K'\subseteq \Rn$. Then every triangulation
$\Delta$ of $K$ has an {\it extension} $\Delta'$ to a triangulation
of $K'$, in the sense that  $\Delta=\{T\in \Delta'\mid T\subseteq K\}.$ 
A continuous
function $f\colon K\to \mathbb R$ is {\it $\Delta$-linear}
if it is linear (in the affine sense) on each simplex of $\Delta.$
Via the extension $\Delta',$  $f$ can be extended
to a $\Delta'$-linear function on $K'.$
A  function $g\colon K\to \mathbb R$ is {\it piecewise linear}
if it is $\Delta$-linear for some triangulation $\Delta$
of $K$. 
We denote by
 $\mathcal R(K)$  the Riesz space
of all  piecewise linear functions on $K$, with the
pointwise operations of the Riesz space  $\mathbb R$.

More generally, let 
$X$ be a nonempty compact subset of $\Rn.$
Let $K\subseteq \Rn$ be an  (always closed)
 $n$-cube containing $X$. We
momentarily denote
by $\mathcal R(K)\restrict X$ the Riesz space
of restrictions to $X$ of the functions in $\mathcal R(K).$
If $L \subseteq \Rn$ is an  
$n$-cube   containing $ K$, then
$\mathcal R(K)\restrict X=\mathcal R(L)\restrict X$.
(For the nontrivial direction, the above mentioned
 extension property of triangulations yields
$\mathcal R(L)\restrict K=\mathcal R(K)$.)
Thus, if  both $n$-cubes $K$  and $L$ contain $X$, letting
$M \subseteq \Rn$
 be an $n$-cube containing both $K$ and $L$, we obtain
$\mathcal R(K)\restrict X=\mathcal R(L)\restrict X=\mathcal R(M)\restrict X$, independently of the ambient cube   $K\supseteq X$. 
Without fear of ambiguity
we may then use the  notation
 $\mathcal R(X)$  for  the Riesz
space of functions thus obtained.
Each  $f\in \mathcal R(X)$ is said to be a
{\it piecewise linear function on $X$.} 
 It follows that   $f$ is continuous. 

\begin{lemma}
\label{lemma:esami}
There is a one-one correspondence
   $x\mapsto \mathfrak m_x,\,\,\,
 \mathfrak m\mapsto x_{\mathfrak m}$ 
 between
 maximal ideals $\mathfrak m$ of $ \mathcal R(X)$  and
 points $x$ of $X$. Specifically, $\mathfrak m_x$
 is the set of all functions in   $ \mathcal R(X)$ vanishing
 at $x$;  conversely,
 $x_\mathfrak m$ is the only element in the intersection of the zerosets
 $Zh=h^{-1}(0)$ of all functions  $h\in \mathfrak m$.
\end{lemma}
\begin{proof}   The functions in
 $ \mathcal R(X)$  separate points, and the constant
 function $1$ is a strong unit in $ \mathcal R(X)$.
 Now apply
 \cite[27.7]{luxzaa}. 
 \end{proof}

The following elementary result deals with 
  the special case  $X\subseteq \mathbb{R}^2$.
It is an adaptation to Riesz spaces 
 of the MV-algebraic result  \cite[Theorem 3.1(ii)]{busmun},
 and will have a key role in the proof of the much stronger Theorem
 \ref{theorem:main}.

\begin{lemma}
\label{lemma:busmun}
Let $X\subseteq \mathbb{R}^2$ be a 
nonempty compact set. If 
 the Riesz space $\mathcal R(X)$
 has a principal
ideal that  is not
an intersection of maximal ideals,  then 
$X$ has  an outgoing Severi-Bouligand  tangent  at 
some point $x\in X$. 
\end{lemma}

\begin{proof}
For every element $e$ of $\mathcal R(X)$ let
$\langle e \rangle$ denote
 the principal ideal generated by~$e$.
Let $g\in \mathcal R(X)$ be such that 
the ideal $\mathfrak p=\langle g \rangle$
 is not an intersection of maximal ideals of $ \mathcal R(X)$.
 Lemma~\ref{lemma:esami} yields 
 an element 
  $f\in \mathcal R(X)$ such that 
  $f\notin \mathfrak p$ and $Zg\subseteq Zf$. 
Replacing, if necessary,
$f$ and $g$ by their absolute values $|f|$ and $|g|$, we
may assume    $f\geq 0$ and $g\geq 0$.
Let  $K  \subseteq \mathbb R^2$
 be a fixed but otherwise arbitrary closed
 square  containing $X$.
By definition of  $\mathcal R(X)$,  there are
elements  
$ 0\leq \tilde{f}\in \mathcal R(K)$
  \,\, and\,\, $0\leq \tilde{g} \in \mathcal R(K)$
 such that $\tilde{f}\restrict X=f$  and 
  $\tilde{g}\restrict X=g$.
Since $\tilde{f}\restrict X$ does not belong to $\mathfrak p$
then for each $m>0$ there is a point $x_m\in X$
 such that 
 \begin{equation}
 \label{equation:sabato}
 \mbox{$\tilde{f}(x_m)>m\cdot \tilde{g}(x_m)$.}
 \end{equation}
 Since $X$ is compact, for some
 $x\in X$
 there is   a subsequence  
  $\{x_{m_1},x_{m_2},\ldots\}$ of $\{x_{1},x_{2},\ldots\}$ such that 
\begin{equation}
\label{equation:nuovo}  
\mbox{$x_i\neq x_j$ for all $i\neq j$,
 \,\,\,  and \,\,\,$\lim_{i\to \infty}x_{m_i}=x$. }
   \end{equation}
For each  $i=1,2,\ldots, $ let  the unit vector  $u_i$
be defined by 
$$u_i=(x_{m_i}-x)/||x_{m_i}-x||.$$
Since the unit circumference
  $S^{1} =\{z\in \mathbb {R}^2\mid ||z||=1\}$ is compact, 
  it is no loss of generality
  to assume   $\lim_{i\to\infty}u_i=u$, for some $u\in S^{1}$. 
  Therefore, $u$ is a tangent of $X$ at $x$.
  There remains to be shown that $u$ is outgoing.
  To this purpose we make the following
  
  \medskip
  \noindent{\it Claim.}  There is a real
  number $\lambda>0$ such that:
\begin{itemize}
\item[(a)] $\tilde{f}$ is (affine) linear on the line segment
 $\conv(x,x+\lambda u)$;
\item[(b)] $\tilde{g}$ identically
vanishes on  $\conv(x,x+\lambda u)$;
\item[(c)] $\tilde{f}(x+\lambda u)\neq 0$.
\end{itemize}

\medskip
As a matter of fact, since
 each of  $x_{m_1},x_{m_2},\ldots$ lies in   $K$, 
 by \eqref{equation:nuovo}
 there exists $\delta>0$ such that $\conv(x,x+\delta u)\subseteq K$.
 An elementary result in polyhedral topology
 (\cite[2.2.4]{sta}) yields
   a triangulation $\Delta$ of $K$ such that 
   both functions
   $\tilde{f}$ and $\tilde{g}$ are $\Delta$-linear
    and 
   $\conv(x,x+\delta u)=\bigcup\{T
   \in\Delta\mid T \subseteq \conv(x,x+\delta u)\}.$
    Therefore, there exists $\lambda>0$ such that $\conv(x,x+\lambda u)\in\Delta$. 
We have proved that
 $\tilde{f}$ is linear in $\conv(x,x+\lambda u)$, and
  (a) is settled.

\medskip To settle (b), since
 both  functions $\tilde{g}$ and $\tilde{f}$ are continuous, we can write
$$
0\geq \tilde{g}(x)\,\,=\,\,\lim_{i\to\infty}\tilde{g}(x_i)
\,\,\leq\,\, \lim_{i\to\infty}\frac{\tilde{f}(x_i)}{m_i}\,=\,0,
$$

\smallskip
\noindent
whence
 $\tilde{g}(x)=g(x)=0$. From
  $X\cap Z\tilde{g}\subseteq X\cap Z\tilde{f}$
  we get  $\tilde{f}(x)=f(x)=0$.
Since $\Delta$ is finite set, there exists
a 2-simplex 
 $S\in\Delta$ containing
 infinitely many   elements
 $x_{n_1},x_{n_2},\ldots$
  of the set  $\{x_{m_1},x_{m_2},\ldots\}$.
    By \eqref{equation:nuovo},
$x\in S$. Further, from
 $\lim_{i\to\infty}u_{n_i}=u$ and $\conv(x,x+\lambda u)\in\Delta$
 it follows that
$\conv(x,x+\lambda u)\subseteq S$.
Therefore, 
\begin{equation}
\label{equation:v}
 \mbox{
 $S=\conv(x,x+\lambda u,v)$ for some $v\in S$.}
\end{equation}
 For some $2\times 1$-matrix
$A$ and vector  $b\in\mathbb{R}^{2}$ we can write
$\tilde{g}(z)=A z+b  \mbox{ for each } z\in S.$
Since $\lim_{i\to\infty}u_{m_i}=u$ and $\tilde{g}(x)=0$, we have
the identities
\begin{align*}
\tilde{g}(x+\lambda u )&=\lambda Au+\tilde{g}(x)=\lim_{i\to \infty} \frac{\lambda(A x_{n_i}- A x)}{||x_{n_i}-x||}=\lim_{i\to \infty} \frac{\lambda(\tilde{g} (x_{n_i})- \tilde{g} (x))}{||x_{n_i}-x||}\\
&=\lim_{i\to \infty} \frac{\lambda \tilde{g} (x_{n_i})}{||x_{n_i}-x||}=\lim_{i\to \infty} \frac{\lambda g(x_{n_i})}{||x_{n_i}-x||}.
\end{align*}
Similarly, 
$$
\tilde{f}(x+\lambda u )=\lim_{i\to \infty} \frac{\lambda f (x_{n_i})}{||x_{n_i}-x||},
$$
whence
\[0\leq \tilde{g}(x+\lambda u )=\lim_{i\to \infty} \frac{\lambda g(x_{n_i})}{||x_{n_i}-x||}\leq \lim_{i\to \infty} \frac{\lambda}{n_i}\frac{ f (x_{n_i})}{||x_{n_i}-x||}=\tilde{f}(x+\lambda u)\lim_{i\to \infty} \frac{1}{n_i}=0.
\]
Since $\tilde{g}$ is    linear 
on $\conv(x,x+\lambda u)$
and $\tilde{g}(x+\lambda u )=0=\tilde{g}(x)$,
then (b) follows.

\medskip
To prove (c),  
by \eqref{equation:sabato} we get 
 $\tilde{f}(x_{n_i})\neq 0$ for all $i$,  whence 
$\tilde{g}(x_{n_i})\neq 0$,
because $Zg\subseteq Zf$.
Then our assumptions about $S$,
together with  \eqref{equation:v},  show that
  $\tilde{g}(v)\neq 0$. Let the integer $m^*$ satisfy   
the inequality  $m^* \cdot
\tilde{g}(v)\geq \tilde{f}(v)$. If
(absurdum hypothesis)
  $\tilde{f}(x+\lambda u)=0$ then $m^* \cdot \tilde{g}(z)\geq 
  \tilde{f}(z)$ for each $z\in S$. In view of
  \eqref{equation:sabato}, this   contradicts
the existence of infinitely many elements $x_{n_i}$ in $S$.
Having thus proved (c), our claim is settled. 

 \smallskip
In conclusion, from  (a) and (c) 
it follows that  $\conv(x,x+\lambda u)\cap Z\tilde{f}=\{x\}$. 
Then from (b) we get  
\[X\cap \conv(x,x+\lambda u)=X\cap Z\tilde{g}\cap \conv(x,x+\lambda u)\subseteq X\cap Z\tilde{f}\cap \conv(x,x+\lambda u)=\{x\}, \]
thus proving that $u$ is an outgoing tangent of $X$ at $x$.
\end{proof}

\section{Tangents and strong semisimplicity}
    
Recall that a Riesz space $R$ is said to be
  {\it strongly semisimple}  if
  for every principal ideal $\langle g\rangle$
of $R$
 the quotient  $R/\langle g\rangle$
is {\it archimedean}  (i.e., the intersection of the maximal ideals
of $R/\langle g\rangle$
 is $\{0\}$).  Equivalently,  $\langle g\rangle$ is an intersection of
maximal ideals of $R$. (This follows from the canonical
one-to-one correspondence  between ideals of $R$
containing $\langle g\rangle$, and ideals of $R/\langle g\rangle$.)
  Since  $\{0\}$ is a principal ideal of $R$,
if $R$ is strongly semisimple then it is archimedean.

\medskip
The following result is the promised 
strengthening of  
 Lemma \ref{lemma:busmun}:

\begin{theorem} 
\label{theorem:main}
For any  nonempty compact  set $X\subseteq \Rn$
the following conditions are equivalent:
\begin{itemize}
\item[(i)]
$X$ has  an outgoing  tangent  at 
some point $x\in X$.  

\smallskip
\item[(ii)]  The Riesz space $\mathcal R(X)$ is not strongly
semisimple, i.e., there exists a principal ideal of 
$\mathcal R(X)$ that is not an intersection of maximal ideals.
\end{itemize}
\end{theorem}

\begin{proof} 
Without loss of generality,  $X\subseteq \I^n.$
(This trivially follows because any $n$-cube in $\mathbb R^n$
is  PL-homeomorphic to any other $n$-cube).

\smallskip
(i)$\Rightarrow$(ii)
By Definition \ref{definition:tangent}, for some  $x\in \Rn$
and $k$-tuple    $u=(u_1,\ldots,u_k)$ 
of pairwise orthogonal unit vectors in $\Rn$, there is a
 sequence $\{x_i\}$ of points in $\Rn$ converging to $x,$
  such that $u$ is the  Frenet $k$-frame of $\{x_i\}$.
 Further, there are    reals $\lambda_1,\ldots,\lambda_k>0$ such that  the simplex
$C=\conv(x,x+\lambda_1u_1,\ldots, x+\lambda_1u_1+\cdots+\lambda_{k}u_{k})$ and its facet $C'=\conv(x,x+\lambda_1u_1,\ldots, x+\lambda_1u_1+\cdots+\lambda_{k-1}u_{k-1})$ 
satisfy  $
C\cap X=C'\cap X.$

Let  $f_1$ and $f_2$ be piecewise linear   
functions defined on $\I^n$, taking their values
in $\mathbb R_{\geq 0}=\{x\in \mathbb R\mid x\geq 0\}$
  and satisfying the conditions
  \begin{equation}
  \label{equation-zerosets}
  \mbox{$Zf_1=f_1^{-1}(0)=C$, \,\,\, $Zf_2=C', $\,\,
and \, $f_2$ is (affine)  linear over $C$.}
\end{equation}
The existence of   $f_1$ and $f_2$
follows from   \cite[2.2.4]{sta}.
Both restrictions
$f_2\restrict X$ and $f_1\restrict X$ are elements of
$ \mathcal R(X).$ By construction,
\begin{equation}
\label{equation:further}
Zf_1\cap X=Zf_2\cap X.
\end{equation}

\medskip
We {\it claim} that the principal ideal $\mathfrak p=\langle
f_1\restrict X\rangle$ of
$ \mathcal R(X)$ 
  generated by $f_1\restrict X$ does not
  coincide with the intersection of all maximal ideals of $ \mathcal R(X)$
 containing $\mathfrak p.$

 \medskip
By  \eqref{equation:further} together with
 Lemma \ref{lemma:esami},
 $f_2\restrict X$ belongs to all maximal
ideals of $\mathcal R(X)$
 containing  $\mathfrak p.$
So our claim will be settled once we prove 
\begin{equation}
\label{equation:there-remains}
 f_2\restrict X\not\in \mathfrak p.
 \end{equation}
To this purpose, arguing by way of contradiction,
suppose    $f_2\restrict X\leq mf_1\restrict X$ for some
$m=1,2,\ldots .$  Since $f_1$ and $f_2$ are (continuous)
piecewise linear,  
the set  $L=\{x\in \I^n \mid f_2(x)\leq mf_1(x) \}$   is a 
union of simplexes  
$T_1\cup\cdots\cup T_r$.
Necessarily for some $j=1,\ldots,r$ the simplex
$T_j$ contains infinitely many points  
of the sequence  $\{x_i\}$. This subsequence
$\{x_t\}$ 
still converges to  $x\in T_j$, and
 $u$ is its  Frenet $k$-frame.
{Theorem~\ref{theorem:contains-u-simplex}} yields     
$\mu_1, \dots , \mu_k>0$ such that 
$T_j$  contains the simplex  
$M  =\conv(x,x+\mu_1u_1,\ldots, x+\mu_1u_1+
\cdots+\mu_ku_k)$.
Now {Proposition  \ref{proposition:filterbase}}
yields
uniquely determined  $\nu_1,\ldots,\nu_k>0$
such that 
$$C\cap M = \conv(x,
          x+ \nu_{1}u_{1},
           \,\ldots, \,\, 
           x+\nu_{1}u_{1}+
           \cdots +\nu_{k}u_{k}).$$
By \eqref{equation-zerosets},   $f_1$
identically  vanishes  on $C\cap M$.
Further, from  $L\supseteq T_j\supseteq M\supseteq
  C\cap M$ and 
 $f_2\leq mf_1$ on $L,$ it follows that $f_2=0$  
   on $C\cap M$. 
The two simplexes $C\cap M$ and
$C$ have the  same dimension $k$, and 
  $f_2$ is (affine)  linear on $C\supseteq C\cap M.$
Therefore, 
  $f_2=0$     on $C$,  which contradicts
 $Zf_2=C'$. We have thus proved \eqref{equation:there-remains},
 settled our claim, 
 and completed the proof of
 (i)$\Rightarrow$(ii).

\bigskip

(ii)$\Rightarrow$(i) 
By hypothesis, there is a function 
$f_1\in \mathcal R([0,1]^n)$  
such that
  the
 principal ideal 
 $\langle f_1\restrict X\rangle$
 of $\mathcal R(X)$ generated by the restriction 
 $f_1\restrict X$ is not an intersection  of  maximal
 ideals  of $\mathcal R(X)$. Thus there is  
$f_2 \in\mathcal R([0,1]^n)$ whose restriction
$f_2\restrict X$
 does not belong to  the
principal ideal $\langle f_1\restrict X\rangle$
generated by  $f_1\restrict X$,   but
 belongs to
 all maximal ideals of $\mathcal R(X)$
 containing $\langle f_1\restrict X\rangle$.
  By Lemma \ref{lemma:esami}, 
 $Zf_2\restrict {X}=Zf_1\restrict X$, i.e., $X\cap Zf_2=X\cap Zf_1$.

 \smallskip

Let the map $g\colon X\to\mathbb{R}^2$ be defined by
\begin{equation}
\label{equation:map}  
g(x)=(f_1(x),f_2(x)) \mbox{ for all } x \in X.
\end{equation}
 Let  
 $\iota\colon  \mathcal R(g(X))\to \mathcal R(X)$ be
defined by $\iota(h)=h\circ g$ for all $h\in \mathcal R(g(X))$,
where $\circ$ denotes composition.
 It is easy to see that $\iota$ is a  Riesz space
  homomorphism of  $ \mathcal R(g(X))$ into $\mathcal R(X)$.
  Letting
   $\pi_1,\pi_2\colon \mathbb{R}^2\to \mathbb{R}$ be
    the  canonical projections  (=coordinate functions), we
    have the identities
    $f_1\restrict {X}=\iota(\pi_1\restrict {g(X)} )$ and $f_2\restrict {X}=\iota(\pi_2\restrict {g(X)} )$. 
Whenever  $h\in  \mathcal R(g(X))$,
 $\iota(h)=0$ and $z\in g(X)$, there exists $x\in X$ 
 such that $g(x)=z$.  Then
  $h(z)=h(g(x))=(\iota(h))(x)=0$ and
     $\iota$ is one-to-one.  Actually, 
     $\iota$ is an isomorphism between 
      $\mathcal R(g(X))$ and the Riesz subspace
       of $\mathcal{R}(X)$ generated by $\{f_1\restrict {X},
       f_2\restrict {X}\}$.
It follows that the principal ideal $\mathfrak  p$  of
$\mathcal R(g(X))$ generated by $\pi_1\restrict  {g(X)}$ is not  an intersection of maximal ideals of $\mathcal R(g(X))$: 
specifically, 
$\pi_2\restrict {g(X)}$ belongs to all maximal ideals
containing $\mathfrak p,$  but does not belong to
$\mathfrak p$.
By Lemma~\ref{lemma:busmun},
\begin{equation}
\label{equation:malta} 
\mbox{\rm $g(X)$  has a 
Severi-Bouligand outgoing  tangent. }
\end{equation}

There remains to be proved that $X$ has an outgoing tangent.
To help the reader, the
 long proof is subdivided into two parts.

\medskip 

$$
\boxed{\mbox{{Part 1:}  {Construction of a tangent $u$
of $X$.}}}
$$

\medskip
\noindent 
 By \eqref{equation:malta} and
  Definition~\ref{definition:tangent} with $k=1$
(which is the same as Definition \ref{definition:severi}), 
for some point $y^* \in\mathbb{R}^2$, 
 unit vector  $v^* \in\mathbb{R}^2$, 
 sequence $\{y_i\}\subseteq \mathbb R^2$
 converging to  $y^*$, 
 and   $\mu>0$, we can write      
   \begin{equation} 
   \label{equation:dopotutto}
   \lim_{i\to \infty}(y_i-y^*)/||y_i-y^*||=v^*
 \,\,\,  \mbox{ and } \,\,\,
   \conv(y^*,y^*+\mu v^*)\cap g(X)=\{y^*\}.
   \end{equation}

\noindent
By \eqref{equation:map}, 
$g$ is the restriction to $X$ of
the function
   $f =(f_1,f_2) \colon \cube\to \mathbb R^2.$
Since (each component of) $f$ is  piecewise linear, 
then $f$ is continuous, and
both  sets $ f^{-1}(y^*)$  and
$ f^{-1}
(\conv(y^*,y^*+\mu v^*))$
are polyhedra in   $\I^n$.
An elementary result in polyhedral topology
(\cite[2.2.4]{sta}) yields a 
 triangulation $\Delta$ of $[0,1]^n$ having the
following properties:

\begin{itemize}
\item 
$f$ is (affine)  linear over each simplex of $\Delta$, 

\smallskip 
\item
 $f^{-1}(y^*)
=\bigcup\{R \in\Delta\mid R \subseteq  f^{-1}(y^*)\}$,
and

\smallskip
\item
$ f^{-1}
(\conv(y^*,y^*+\mu v^*))
=\bigcup\{U\in\Delta\mid U\subseteq  f^{-1}
(\conv(y^*,y^*+\mu v^*)) \}.$ 
\end{itemize}
For some
 $n$-simplex $T\in \Delta$, the set 
$\{i\mid f^{-1}(y_i)\cap T\cap X\}=
\{i\mid g^{-1}(y_i)\cap T\}$
 is infinite. 
Let $z_0,z_1,\ldots$ be a converging sequence of 
elements of $T$ such that 
$f(z_0),f(z_1),\ldots$
 is a subsequence of $y_0,y_1,\ldots$.
 Without loss of generality this subsequence
 coincides with the sequence $\{y_i\}$, and we
 can write
 \begin{equation}
 \label{equation:z-y}
 g(z_i)=y_i.
 \end{equation}
 Letting   $z^*=\lim_{i\to \infty }z_i$ we have
   \begin{equation}
   \label{equation:lorena}
 z^*\in X\cap T
 \,\,\mbox{  and  }\,\, y^*= f(z^*)=g(z^*).
   \end{equation}
The  linearity of $f$  on $T$  yields
a  $2\times n$   matrix  $A$,  together with 
a vector   $b\in \mathbb R^2$
such that
$
\mbox{ for each } t\in T,\,\,\, f(t)=At+b.
$

 \medskip
  \noindent{\it Claim.}
For some $k\in \{1,\ldots,n\}$  there is a $k$-tuple
of  pairwise orthogonal unit
vectors  $u_i \in \mathbb R^n,\,\,\, (1\leq i\leq k)$ 
such that:
\begin{itemize}
\item
  $Au_j=0$  for each $1\leq j<k$,  
\item
 $Au_k\neq 0$, 
\item
  $u=(u_1,\ldots,u_k)$  is a tangent of $X$ at $z^*$, 
  determined by a suitable subsequence of $z_0,z_1,\ldots$,
  in the sense of Definition \ref{definition:tangent}.
  \end{itemize}

\medskip
The vectors  $u_1,\ldots,u_k$  are constructed 
by the following inductive procedure:

\medskip
\noindent{\it Basis Step:} 
From  $Az_i +b=y_i\neq y^*=Az^*+b$ 
it follows that  $z_i\neq z^*$ for each $i$,  and hence  
every vector  $z_i^{1}=(z_i-z^*)/||z_i-z^*||$ 
is well defined. 
Since the $(n-1)$-dimensional unit sphere
$S^{n-1}\subseteq \Rn$ is compact, it is no
loss of generality to assume that  
 the sequence
$z_0^{1},z_1^{1},\ldots$ converges to some unit 
vector $u_{1}$. 
It follows that $u_1$ is a tangent of $X$ at $z^*$. 
If $Au_1\neq 0$,  upon setting $u=u_1$
the claim is proved.
If $Au_1= 0$ we proceed inductively.

\medskip
\noindent{\it Induction Step:} 
Having  constructed 
a tangent $u(l)=(u_1,\ldots,u_l)$ of $X$ at $z^*$ with $Au_i=0$ for each $i\in\{1,\ldots,l\}$, we first observe that
  $l<n$. (For otherwise,  the $u_j$ 
  would constitute an
  orthonormal basis of $\Rn$, whence 
    $A$  is the zero matrix, and  $Ax+b=b$ for each $x\in \Rn, $ 
   which contradicts $Az_i+b\neq Az^*+b$.)
Let  $\rho_1,\ldots,\rho_l$ be arbitrary real numbers. From 
\begin{equation}
\label{equation:identities}
 A(z^*+\rho_1 u_1+\cdots+\rho_l u_l)+b
 =A(z^*)+b=g(z^*)\neq g(z_i)=A(z_i)+b,
\end{equation}
it  follows that   no
$z_i$ lies  in the affine space
$z^*+\mathbb{R}u_1+\cdots +\mathbb{R}u_l$,
i.e.,  $z_i-z^*\notin\mathbb{R}u_1+\cdots +\mathbb{R}u_l $.
For each $i$, the unit vector
 $$
 z_i^{l+1}=
 \frac{z_i-z^*-\mathsf{proj}_{\mathbb{R}u_1
 +\cdots +\mathbb{R}u_l} (z_i-z^*)}
 {||z_i-z^*-
 \mathsf{proj}_{\mathbb{R}u_1+\cdots +\mathbb{R}u_l}(z_i-z^*)||}
 $$
is well defined.
Without loss of generality, we can write 
 $\lim_{i\to\infty}z_i^{l+1}=u_{l+1}$  
for some unit 
 vector $u_{l+1}\in \Rn$. 
By construction, 
 $u_{l+1}$ is orthogonal to each of $u_1,\ldots,u_{l}$,   
 and the $(l+1)$-tuple
 $u(l+1)=(u_1,\ldots,u_l,u_{l+1})$
  is a tangent of $X$ at $z^*$. 
In case  $Au_{l+1}\neq0$, upon setting $k=l+1$
and  $u=u(l+1)$ 
we are done.   
 In case  $Au_{l+1}= 0,$
we proceed inductively, with
  $(u_1,\ldots,u_l,u_{l+1})$ in place
  of $(u_1,\ldots,u_l)$.  Our claim
  is settled, and so is the proof
of Part 1.

$$
\boxed{\mbox{{Part 2:}  {  $u$ is an outgoing tangent of $X$.}}}
$$

\smallskip  
\noindent With the notation
of Part 1,
 for some
 $\lambda_1,\ldots,\lambda_k>0$ we prove the inclusion 
 \begin{equation}
 \label{equation:claim-two}
\conv(z^*,z^*+\lambda_1u_1,\ldots, z^*+\lambda_1u_1+\cdots+\lambda_{k}u_{k})\subseteq  T\cap f^{-1}
(\conv(y^*,y^*+\mu v^*)).
\end{equation}

As a matter of fact, by construction,
 $u=(u_1,\ldots, u_k)$ is  a  tangent of $X\cap T$ at~$z^*$.
{Theorem \ref{theorem:contains-u-simplex}}
yields  real numbers $\epsilon_1,\ldots,\epsilon_k>0$ 
 such that 
 \begin{equation}
 \label{equation:epsilon-T}
 \conv(z^*,z^*+\epsilon_1u_1,\ldots, z^*+\epsilon_1u_1+\cdots+\epsilon_{k}u_{k})\subseteq T.
 \end{equation} 
Since   $Au_j=0$ for each $j=1,\ldots,k-1$, from
\eqref{equation:lorena}-\eqref{equation:identities} we 
obtain the identities
\begin{equation} 
\label{equation:lorrain}
 y^*=g(z^*)=g(x)
\mbox{
for all }
 x\in \conv(z^*,z^*+\epsilon_1u_1,\ldots, z^*+\epsilon_1u_1+\cdots+\epsilon_{k-1}u_{k-1}).
 \end{equation} 
Recalling  \eqref{equation:z-y} we can write
\begin{align*}
0\neq Au_k&
=\lim_{i\to\infty}Az^k_i
=\lim_{i\to\infty}A\left(\frac{z_i-z^*-\mathsf{proj}_{\mathbb{R}u_1+\cdots +\mathbb{R}u_{l-1}}(z_i-z^*)}{||z_i-z^*-\mathsf{proj}_{\mathbb{R}u_1+\cdots +\mathbb{R}u_{l-1}}(z_i-z^*)||}\right)\\[0.2cm]
	&=\lim_{i\to\infty}\frac{A(z_i)-A(z^*)}{||z_i-z^*-\mathsf{proj}_{\mathbb{R}u_1+\cdots +\mathbb{R}u_{l-1}}(z_i-z^*)||}\\[0.2cm]
&=\lim_{i\to\infty}\frac{y_i-y^*}{||z_i-z^*-\mathsf{proj}_{\mathbb{R}u_1+\cdots +\mathbb{R}u_{l-1}}(z_i-z^*)||}
\cdot \frac{||y_i-y^*||}{||y_i-y^*||}\\[0.2cm]
&=\lim_{i\to\infty}\frac{y_i-y^*}{||y_i-y^*||}\cdot
\frac{||y_i-y^*||}{||z_i-z^*-\mathsf{proj}_{\mathbb{R}u_1+\cdots +\mathbb{R}u_{l-1}}(z_i-z^*)||}.
\end{align*}

\bigskip
\noindent
Since 
$0\not= v^* =
  \lim_{i\to\infty}(y_i-y^*)/||y_i-y^*||,$
for some   $\tau >0$ we obtain

$$ \tau=\lim_{i\to \infty}
 \frac{||y_i-y^*||}{||z_i-z^*-\mathsf{proj}_{\mathbb{R}u_1+\cdots +\mathbb{R}u_{l-1}}(z_i-z^*)||}\,
 \,\,\,\mbox{  and   }\,\,Au_k=\tau v^*.
$$

\smallskip
\noindent
Now the desired $\lambda$'s
in  \eqref{equation:claim-two} are given  by setting
$\lambda_j=\epsilon_j$ for $1\leq j< k$,\,\,and\,\,
$\lambda_k=\min\{\epsilon_k,\mu/\tau\}.$
Indeed, letting   
$C=\conv(z^*,z^*+\lambda_1u_1,\ldots, z^*+\lambda_1u_1+\cdots+\lambda_{k}u_{k}),$
from  \eqref{equation:epsilon-T} we  obtain
 \begin{equation}
 \label{equation:ct}
C\subseteq 
\conv(z^*,z^*+\epsilon_1u_1,\ldots, z^*+\epsilon_1u_1+\cdots+\epsilon_{k}u_{k})
\subseteq T.
\end{equation}
Further, for every  $x\in C$
there exists  $0\leq \omega\leq\lambda_k$ such that
\begin{equation}
\label{Eq:Azk}
Ax+b=Az^*+\omega Au_k +b=Az^*+b+\omega \tau v^*=y^*+\omega \tau v^*,
\end{equation}
whence
  $Ax+b\in \conv(y^*,y^*+\mu v^*)$, because  $\omega\leq \mu/\tau$.
The proof of \eqref{equation:claim-two} is complete.

\bigskip
To complete the proof that  
$(u_1,\ldots,u_k)$ is outgoing, letting 
$C'=\conv(z^*,z^*+\lambda_1u_1,\ldots, z^*+\lambda_1u_1+\cdots+\lambda_{k-1}u_{k-1}),$
we must  show $C'\cap X=C\cap X.$
By way of contradiction, suppose
$ x \in (X\cap C)\setminus (X\cap C')$.   
Then for suitable
$\xi_1,\ldots,\xi_{k-1}\geq 0\,\,\,$ and $\,\,\xi_k>0, $
we can write  $x=z^*+\xi_1u_1+\cdots+\xi_{k}u_{k}.\,\,$
By \eqref{equation:ct},  $\,\,x \in X\cap T$. 
Since $\xi_k>0$, by \eqref{Eq:Azk}  we have
$
g(x)=f(x)=Ax+b=y^*+\xi_k \tau v^*\neq y^*.
$
This contradicts the identity
$g(x)\in g(X)\cap \conv(y^*,y^*+\mu v^*)=\{y^*\},\,\,$
which follows from  \eqref{equation:dopotutto} and~\eqref{equation:lorrain}.

\bigskip
Having thus  proved that the tangent
$u$ is outgoing,  we have also
completed the proof of Part 2, as well as the
  proof of the theorem. 
\end{proof}

\section{Examples and Further Results}
\begin{proposition}
Let $I=\conv(a,b)\subseteq \mathbb R$
be an interval, and 
$\phi\colon I \to \Rn$  a $C^{2}$ function.
Then the Riesz space
 $\mathcal R(\phi(I))$ is strongly semisimple iff $\phi$ is
 (affine)  linear.
\end{proposition}
\begin{proof} The proof directly follows from
   Theorems~\ref{theorem:main} and~\ref{theorem:eguali}.
\end{proof}

\begin{proposition}
\label{proposition:Poly}
For every polyhedron
 $P\subseteq \Rn$ the Riesz space
  $\mathcal R (P)$ is strongly semisimple,
  and $P$ has no outgoing tangent. 
\end{proposition}
\begin{proof}
For some finite set 
$\{S_1,\ldots,S_m\}$
 of simplexes in $\Rn$ we can write 
$P=S_1\cup\cdots\cup S_m$.
If  $u$ is a tangent of $P$ at some point $x\in P$
then     $\,\,u$  is also 
  a tangent  
  of  $S_i$ at $x$   for some $i=1,\dots,m$.
 By Theorem~\ref{theorem:contains-u-simplex},  
$u$ is not an outgoing tangent of $S_i$.
Thus  $u$ is not an outgoing tangent of $P$. Now 
apply Theorem~\ref{theorem:main}.
\end{proof}

The following is an example of
a strongly semisimple Riesz space $\mathcal R(X)$, 
 where $X$ is not a polyhedron:

\begin{example}
\label{counterexample}
Let the set  $X\subseteq \mathbb R^{2}$ be defined by 
\[
X=\textstyle\{(0,0)\}\cup\{({1}/{n},0)
\mid n = 1,2,\ldots\,\}\cup \{({1}/{n},{1}/{n^2})
\mid n = 1,2,\ldots\,\}.\]
The origin  $(0,0)$ is the only accumulation point of $X$.
The only tangents of $X$ 
 are given by the vector
$(1,0)$ and the pair of vectors $((1,0),(0,1))$. Therefore, $X$ has no outgoing tangents. By  Theorem~\ref{theorem:main}, the
Riesz space
 $\mathcal R(X)$ is strongly semisimple.
\end{example}
However, when the compact set
$X\subseteq\Rn$ is convex we have: 
\begin{theorem}
\label{theorem:convex}
Let  $X\subseteq\Rn$ be a nonempty compact convex set.
Then the following conditions are equivalent:
\begin{itemize}
\item[(I)]
The Riesz space $\mathcal R (X)$ is strongly semisimple.

\item[(II)]
  $X=\conv(x_1,\ldots,x_m)$ for some $x_1,\ldots,x_m\in \Rn$, i.e.,
  $X$ is a polyhedron.

\item[(III)]
$X$ has no outgoing tangent.

\end{itemize}
\end{theorem}

\begin{proof}
(III)$\Leftrightarrow$(I)  This is a particular case of Theorem
\ref{theorem:main}. 
(II)$\Rightarrow$(I)
By  Proposition~\ref{proposition:Poly}.
(I)$\Rightarrow$(II)
Arguing by way of contradiction,
assume
$\mathcal R (P)$ to be strongly semisimple, but
  $X\neq\conv(x_1,\ldots,x_m)$  for any 
  finite set $\{x_1,\ldots,x_m\} \subseteq  \Rn$. 
   Letting
  ${\rm ext}( X )$ denote the set of extreme point of $X$,
  Minkowski theorem yields the identity
    $X=\conv({\rm ext}( X ))$.
Since   $X$ is compact, there exists
a point $x\in X$ together with
 a sequence $x_1,x_2,\ldots$ of extreme points of $X$
 such that $\lim_{i\to \infty} x_i=x$ and $x_i\neq x_j$ for every $i\neq j$.

\bigskip
 \noindent
{\it Claim 1.} There exists a subsequence $x_{m_1},x_{m_2},\ldots$ of the sequence $x_1,x_2,\ldots$, together with 
a $k$-tuple $(u_1,\ldots,u_{k})$
of pairwise orthogonal unit vectors
 in $\Rn$
 (for some $k\in\{1,\dots,n\}$), 
  having the following properties: 
 
\begin{itemize}
\item[(a)] $x_{m_1},x_{m_2},\ldots$ 
determines the tangent $(u_1,\ldots,u_{k})$ of $X$ at $x$, 
in the sense of Definition \ref{definition:tangent}.

\smallskip
\item[(b)] $\aff(x_{m_1},x_{m_2},\ldots)=x+\mathbb Ru_1+\cdots+\mathbb Ru_{k}.$
\end{itemize}

\bigskip
The  vectors  $u_1,u_2,\dots,u_k$ are constructed by the following
inductive procedure:  

\smallskip

\noindent{\it Basis:} 
Since $x_i\neq x_j$ for each $i\neq j$,  then
each unit vector $({x_i-x})/{||x_i-x||}$ is well defined.
 There is a subsequence $x_{m^1_1},x_{m^1_2},\ldots$ of  $x_1,x_2,\ldots$ and a unit vector $u_1\in\Rn$ such that 
$
\lim_{i\to\infty}
({x_{m^1_i}-x})/{||x_{m^1_i}-x||} =u_1.
$
Then $u_1$ is a tangent of $X$ at $x$ 
determined by $x_{m^1_1},x_{m^1_2},\ldots$.
\medskip

\noindent{\it Induction Step:}
Let    $ l\geq 1$ and assume the  
subsequence $x_{m^l_1},x_{m^l_2},\ldots$ of  
$x_1,x_2,\ldots$ determines
the tangent $(u_1,\ldots,u_{ l})$ of $X$ at $x$. 
If there exists an integer $r$ such that 
$\aff(x_{m^{ l}_r},x_{m^{ l}_{r+1}},\ldots)=x+\mathbb Ru_1+\cdots+\mathbb Ru_{ l},$ then upon setting $k=l$, we are done. 
If no such $r$ exists,
  infinitely many vectors in $x_{m^ l_1},x_{m^ l_2},\ldots$ 
do not belong to the affine space 
$x+\mathbb Ru_1+\cdots+\mathbb Ru_{ l}$. 
Therefore, for some
 subsequence  $x_{m^{ l+1}_1},x_{m^{ l+1}_{2}},\ldots$ and  unit vector $u_{ l+1}\in\Rn$ 
 we can write
\begin{equation}
\label{equation:due}
u_{ l+1}=\lim _{i\to \infty}\frac{
x_{m^{l+1}_i}-x-
\mathsf{proj}_{\mathbb Ru_1+\cdots+\mathbb Ru_{l}}(x_{m^{ l+1}_i}-x)
}
{||x_{m^{l+1}_i}-x-
\mathsf{proj}_{\mathbb Ru_1+\cdots+\mathbb Ru_{l}}(x_{m^{ l+1}_i}-x)||}.\end{equation}
We then proceed with $(u_1,\ldots,u_{l+1})
$ in place of $(u_1,\ldots,u_{ l})$.
Since the affine space $\aff(x_{m_1},x_{m_2},\ldots)$ is 
contained in $ \Rn$, this
procedure must terminate for some  $1\leq k\leq n$.
Claim 1 is settled.

\bigskip

 Let us now fix a subsequence \,\,$x_{m_1},x_{m_2},\ldots$\,\, of \,\,$x_1,x_2,\ldots,$\,\, 
 together with a $k$-tuple $(u_1,\ldots,u_{k})$
of pairwise orthogonal unit vectors satisfying
conditions  (a) and (b) in Claim 1.

\bigskip

\noindent{\it  Claim 2.} There are  
$\lambda_1,\ldots, \lambda_k>0$  such that  the
$k$-simplex
$
C_k=\conv(x,x+\lambda_1u_1,\ldots, x+\lambda_1u_1+\cdots+\lambda_k u_k)\mbox{ is contained in } X.
$

\medskip
We have already observed that $x\in X$. By Theorem~\ref{theorem:main}, the tangent $u_1$ of $X$ at $x$ is not outgoing. Hence $\conv(x,x+u_1)\cap X\neq\{x\}$. Let $y\in (\conv(x,x+u_1)\cap X)\setminus\{x\}$.  Thus
 $y=x+\lambda_1 u_1$ for some $0<\lambda_1\leq 1$. 
 Since $X$ is convex, $\conv(x,x+\lambda_1 u_1)\subseteq X$.

Proceeding inductively, let us assume that 
  $\lambda_1,\ldots, \lambda_l>0$  are such that 
  the $l$-simplex
$
C_l=\conv(x,x+\lambda_1u_1,\ldots, x
+\lambda_1u_1+\cdots+\lambda_l u_l)
\mbox{ is contained in } X,
$
   for some $l\in\{1,\dots,k\}$. 
If $l=k$ we are done. If $l<k$
     let
$
C'_{l+1}=\conv(x,x+\lambda_1u_1,\ldots, x+\lambda_1u_1+\cdots+ \lambda_l u_l+ u_{l+1}).
$
By construction, 
     $(u_1,\ldots,u_{l+1})$ is a tangent of $X$ at $x$. 
 Since by hypothesis $\mathcal R(X)$ is strongly semisimple,  
 by  Theorem~\ref{theorem:main}  
  $(u_1,\ldots,u_{l+1})$ is not outgoing, whence
  there is  
  $
  y\in (C'_{l+1}\cap X)\setminus C_l.
  $
As a consequence, there 
are $\lambda'_1,\ldots, \lambda'_l>0$ and $\lambda_{l+1}>0$
such that   
$
y=x+\lambda'_1u_1+\cdots+ \lambda'_l u_l+\lambda_{l+1} u_{l+1}
\mbox{ and }\lambda'_i\leq \lambda_i 
\mbox{ for each } i\in\{1,\dots,l\}.
$
Since $X$ is convex,   the set 
$
\conv(x,\, x+\lambda'_1u_1,\ldots, x+\lambda'_1u_1
+\cdots+\lambda'_l u_l,\,\,y)
$ is contained in $X$. 
Setting now  (without loss of generality)
 $\lambda_i=\lambda_i',$ we obtain the inclusion
$
C_{l+1}=\conv(x,x+\lambda_1u_1,\ldots, x
+\lambda_1u_1+\cdots+\lambda_{l+1} u_{l+1})
\subseteq  X,
$
thus completing the inductive step.
This procedure terminates after  $k$ steps.  Claim 2 is settled.

\medskip


Since the $k$-simplex
$C_k$ is contained in the affine space 
$\,\, \aff(x_{m_1},x_{m_2},\ldots),\,\,$  and 
$(u_1,\ldots,u_k)$ is the Frenet $k$-frame of the sequence
$x_{m_1},x_{m_2},\ldots$,
the exists
  an integer $r^*>0$ such that 
$x_{m_j}\in C_k$ for each $j= r^*, r^*+1,\dots$.
 By definition,   $x_{m_1},x_{m_2},\ldots\in {\rm ext} (X)$.  
 By Claim 2, $C_k\subseteq X$.
Thus  
  $x_{m_{r^*}},x_{m_{r^*+1}},\ldots\in  {\rm ext} (C_k)$. 
Since $x_i\neq x_j$ for every $i\neq j$, then the set 
 ${\rm ext} (C_k)$ must be infinite, a contradiction.
The proof is complete.
\end{proof}

\bibliographystyle{plain}

  \end{document}